\documentclass[12pt]{amsart}

\usepackage{amsthm,amsfonts,amsmath,amssymb,latexsym,epsfig}
\usepackage[usenames]{color}

\DeclareMathAlphabet\oldmathcal{OMS}        {cmsy}{b}{n}
\SetMathAlphabet    \oldmathcal{normal}{OMS}{cmsy}{m}{n}
\DeclareMathAlphabet\oldmathbcal{OMS}       {cmsy}{b}{n}
 
\usepackage{eucal}

\usepackage[active]{srcltx}
\usepackage{hyperref}

\allowdisplaybreaks

\newtheorem{theorem}{Theorem}
\newtheorem{lemma}[theorem]{Lemma}
\newtheorem{proposition}[theorem]{Proposition}

\newtheorem{definition}[theorem]{Definition}
\newenvironment{example}{\medskip \refstepcounter{theorem}
\noindent  {\bf Example \thetheorem}.\rm}{\,}
\newenvironment{remark}{\medskip \refstepcounter{theorem}
\noindent  {\bf Remark \thetheorem}.\rm}{\,}
\newtheorem*{ack}{Acknowledgements}

\renewcommand{\thetheorem}{\thesection.\arabic{theorem}}
\def\d{\partial}

\def\<{\langle}

\def\>{\rangle}

\def\w{\wedge}                                     

\def\BOne{{\mathchoice {\rm 1\mskip-4mu l} {\rm 1\mskip-4mu l}
                          {\rm 1\mskip-4.5mu l} {\rm 1\mskip-5mu l}}}

\def\fract#1#2{\raise4pt\hbox{$ #1 \atop #2 $}}
\def\decdnar#1{\phantom{\hbox{$\scriptstyle{#1}$}}
\left\downarrow\vbox{\vskip15pt\hbox{$\scriptstyle{#1}$}}\right.}

\def\bbc{{\mathbb C}}

\def\bbn{{\mathbb N}}

\def\bbp{{\mathbb P}}
\def\bbq{{\mathbb Q}}
\def\bbr{{\mathbb R}}

\def\bbz{{\mathbb Z}}

\def\grl{\lambda}

\def\grS{\Sigma}

\def\bfa{{\bf a}}

\def\bfw{{\bf w}}

\def\bfz{{\bf z}}

\def\calo{{\mathcal O}}

\def\cald{{\mathcal D}}

\def\calf{{\mathcal F}}

\def\calo{{\mathcal O}}

\def\cals{{\oldmathcal S}}

\def\calz{{\oldmathcal Z}}

\def\la#1{\hbox to #1pc{\leftarrowfill}}
\def\ra#1{\hbox to #1pc{\rightarrowfill}}

\def\gc{{\mathfrak c}}

\def\gf{{\mathfrak f}}

\def\gi{{\mathfrak i}}

\def\gt{{\mathfrak t}}
\def\gu{{\mathfrak u}}

\def\gA{{\mathfrak A}}

\def\gD{{\mathfrak D}}

\def\gF{{\mathfrak F}}

\def\gM{{\mathfrak M}}

\def\gP{{\mathfrak P}}

\def\dif{\gD\gi\gf\gf(M)}

\def\hook{\mathbin{\hbox to 6pt{%
                 \vrule height0.4pt width5pt depth0pt
                 \kern-.4pt
                 \vrule height6pt width0.4pt depth0pt\hss}}}

\def\lcm{{\rm lcm}}
\def\coker{{\rm coker}}

\def\id{{\rm Id}}
\def\rk{{\rm rank\,}}

\begin{document}
\date{\today}
\title{Brieskorn Manifolds, Positive Sasakian Geometry, and Contact Topology}
\author{Charles P. Boyer, Leonardo Macarini, and Otto van Koert}\thanks{The first author was partially supported by a grant (\#245002) from the Simons Foundation. The second author was partially supported by CNPq, Brazil. The third author was supported by a stipend from the Humboldt foundation.
}
\address{Charles P. Boyer, Department of Mathematics and Statistics,
University of New Mexico, Albuquerque, NM 87131.}
\email{cboyer@math.unm.edu} 
\address{Leonardo Macarini, Universidade Federal do Rio de Janeiro, Instituto de Matem\'atica, Cidade Universit\'aria, CEP, 21941-909 - Rio de Janeiro Brazil} \email{leonardo@impa.br}
\address{Otto van Koert, Department of Mathematics and Research Institute of Mathematics, Seoul National University,
Building 27, room 402, San 56-1, Sillim-dong, Gwanak-gu, Seoul, South Korea, Postal code 151-747}
\email{okoert@snu.ac.kr}

\maketitle

\markboth{Positive Sasakian Geometry and Contact Topology}{Charles P. Boyer, Leonardo Macarini, and Otto van Koert}

\begin{abstract}
Using $S^1$-equivariant symplectic homology, in particular its mean Euler characteristic, of the natural filling of links of Brieskorn-Pham polynomials, we prove the existence of infinitely many inequivalent contact structures on various manifolds, including in dimension $5$ the $k$-fold connected sums of $S^2\times S^3$ and certain rational homology spheres. We then apply our result to show that on these manifolds the moduli space of classes of positive Sasakian structures has infinitely many components. We also apply our results to give lower bounds on the number of components of the moduli space of Sasaki-Einstein metrics on certain homotopy spheres. Finally a new family of Sasaki-Einstein metrics of real dimension $20$ on $S^5$ is exhibited.
\end{abstract}

\section{Introduction}
Brieskorn manifolds have played an important role in the development of both Sasakian geometry and contact topology. On the one hand they provide examples of Sasaki-Einstein metrics as well as Sasaki metrics of positive Ricci curvature on many manifolds, in particular on spheres including exotic ones \cite{BGN03b,BG05h,BGK05,BG05}. On the other hand they provide examples of distinct contact structures with the same classical invariants \cite{Ust99,vKoe05,vKoe08,KwvKo13,Fauck:RFH,Ueb15}. 

The purpose of this note is to use the $S^1$-equivariant symplectic homology (in particular, the mean Euler characteristic) of Liouville fillings of Brieskorn manifolds to obtain connectivity results about certain moduli spaces of Sasakian structures. Kwon and van Koert \cite{KwvKo13} have shown that every rational number can be realized as the mean Euler characteristic of some contact structure on the 5-sphere $S^5$. Since the mean Euler characteristic is a contact invariant, this implies that there are infinitely many distinct contact structures on $S^5$. However, not all of these admit a compatible Sasaki metric, but there are infinitely many that do, namely those coming from the Brieskorn manifold with exponents $(2,2,p,q)$ where $p,q$ are relatively prime positive integers. Indeed, these all have positive Sasakian structures. Thus, the moduli space of deformation classes of positive Sasakian structures on $S^5$ has infinitely many components. On the other hand, Koll\'ar \cite{Kol05b} has shown that certain 5-dimensional rational homology spheres have precisely one component of the moduli space of deformation classes of positive Sasakian structures. In fact, many five dimensional rational homology spheres do not carry a positive Sasakian structure at all, although they can admit other Sasakian structures.

In Section \ref{defsassec} we define the moduli space $\gM_{+,0}^c(M)$ of positive Sasaki classes with $c_1(\cald)=0$ on a given manifold $M$ of Sasaki type. By the transverse version of Yau's Theorem, each element of $\gM_{+,0}^c(M)$ has a Sasaki metric with positive Ricci curvature. One of the main purposes of this paper is to understand the set $\pi_0(\gM_{+,0}^c(M))$ of components of $\gM_{+,0}^c(M)$. As mentioned above, we know from \cite{KwvKo13} that $|\pi_0(\gM_{+,0}^c(S^5))|=\aleph_0$. We first consider dimension 5 where we employ Smale's classification and notation \cite{Sm62} of simply connected spin manifolds (note that a simply connected 5-manifold admitting contact structures with $c_1(\cald)=0$ must be spin). Here $M_k$ denotes the rational homology $5$-sphere with $H_2(M_k,\bbz)\approx \bbz_k\oplus \bbz_k$, $lM_k$ denotes the $l$-fold connected sum of $M_k$, and $k(S^2\times S^3)$ denotes the $k$-fold connected sum of $S^2\times S^3$. We now describe our main results:

\begin{theorem}\label{ks2s3comp}
For each $k=1,2,\dots$ we have $|\pi_0(\gM_{+,0}^c(k(S^2\times S^3)))|=\aleph_0$. Moreover, each component belongs to a distinct contact structure, so there are infinitely many inequivalent contact structures of positive Sasaki type with vanishing first Chern class on $k(S^2\times S^3)$ for each such $k$.
\end{theorem}

While, as mentioned above, for certain rational homology spheres $M$, $\gM_{+,0}^c(M)$ has precisely one component, on the other hand we have

\begin{theorem}\label{rathomsphcomp}
For the rational homology 5-spheres 
$$M=M_2,M_3,M_5,2M_3,4M_2$$ 
we have $|\pi_0(\gM_{+,0}^c(M))|=\aleph_0$. Moreover, each component belongs to a distinct contact structure, so there are infinitely many inequivalent contact structures of positive Sasaki type on each of the above rational homology 5-spheres.
\end{theorem}

These are the only nontrivial rational homology $5$-spheres realized as the link of a Brieskorn-Pham (BP) polynomial with infinitely many inequivalent contact structures of positive Sasaki type. For a complete discussion of Koll\'ar's results on positive Sasakian structures in dimension 5 we refer to 10.2.1 of \cite{BG05} as well as the original paper \cite{Kol05b}. Further results can also be found in \cite{BN09}. However, these results are beyond the scope of the present paper in that they do not arise from BP polynomials. A classification of positive Sasakian structures on 5-manifolds represented by BP polynomials is found in the tables of Appendix B.4 of \cite{BG05}.

For dimension one mod 4 and larger than 5 we consider the manifolds
$$S^{2n}\times S^{2n+1},~S^{2n}\times S^{2n+1}\#\grS^{4n+1},~\text{and} ~T,$$
where $\grS^{4n+1}$ is a generator of the group $bP_{4n+2}$ of exotic spheres that bound parallelizable manifolds of dimension $4n+2$, and $T$ is the unit tangent sphere bundle over $S^{2n+1}$. We remark that $bP_{4n+2}$ is the identity for $n=1,3,7,15$ and it is $\bbz_2$ when $4n+2\neq 2^i-2$ for $i\geq 3$, otherwise it is unknown. It is shown in \cite{BG05h} that these manifolds admit positive Sasaki metrics. Concerning their moduli we now have

\begin{theorem}\label{highdim}
The following hold:
\begin{align*}
& |\pi_0(\gM_{+,0}^c(S^{2n}\times S^{2n+1}))| = \aleph_0, \\
& |\pi_0(\gM_{+,0}^c(S^{2n}\times S^{2n+1}\#\grS^{4n+1}))| =  \aleph_0, \\
& |\pi_0(\gM_{+,0}^c(T))| =\aleph_0.
\end{align*}
Moreover, each component belongs to a distinct contact structure, so there are infinitely many inequivalent contact structures of positive Sasaki type on each of the above smooth manifolds.
\end{theorem}


We now want to consider BP links which admit Sasaki-Einstein metrics on homotopy spheres. (For convenience we exclude the standard round sphere from any discussion in this paper. Thus, all contact structures on spheres considered here are exotic.) However, owing to the lack of a completeness theorem a la Kodaira-Spencer in the weighted polynomial case, we only obtain definitive results for the subspace of Sasaki-Einstein metrics that are represented by distinct weighted homogeneous polynomials $\gM^{WHSE}\subset \gM^{SE}$. (See the discussion on pages 177-178 of \cite{BG05}). See also Chapter 11 of \cite{BG05} and references therein. The homotopy spheres we consider here are $S^5,S^7,S^9$. In particular, we obtain lower bounds on the number of components of the local moduli space $\gM^{SE}$ of each diffeomorphism type. It is unknown whether the number of components of $\gM^{WHSE}$ is actually finite, but only finitely many are known on a homotopy sphere of any dimension. The 5-sphere $S^5$ has a unique diffeomorphism type and until now there were $81$ known families of SE metrics, not including the standard round sphere, given by $80$ families established in \cite{BGK05,GhKo05} plus the Li-Sun example (see Remark \ref{SE2}). However, we have found a new family of SE metrics with $20$ real parameters given by Theorem \ref{newSE} below, giving a total of $82$. For these $82$ families we have an Excel table that shows that there are $7$ pairs whose contact structures cannot be distinguish by the mean Euler characteristic. However, one pair can be distinguished by their $S^1$ equivariant symplectic homology, see Lemma~\ref{lemma:examples_differ}. Therefore, we have

\begin{theorem}\label{s5SEthm}
$|\pi_0(\gM^{SE}(S^5))|\geq 76$. Concerning $\gM^{WHSE}$ there are $55$ components consisting of single points, eighteen components of real dimension $2$, four components of real dimension $4$, one component of real dimension $6$, two components of real dimension $8$, one component of real dimension $10$, and one component of real dimension $20$. Moreover, the $82$ components of $\gM^{WHSE}$ belong to at least $76$ distinct components of $\gM^c_{+,0}(S^5)$.
\end{theorem}

Of the $6$ pairs of elements of $\gM^{WHSE}$ whose contact structures cannot be distinquished, one has both members with zero dimensional moduli and one has both members with 2 real dimensional moduli. The remaining four cases consist of a pair of elements of $\gM^{WHSE}$ with real dimensions two and six, two and zero, eight and zero, and zero and four, respectively. The elements of a pair probably belong to distinct components of $\gM^{SE}$, but we do not have a proof of this at this time. On the other hand there is no reason to believe one way or the other whether they belong to distinct elements of $\gM^c_{+,0}(S^5)$ or not. 

In the case of homotopy $S^7$s and $S^9$s, we do not determine the dimension of the various components, although in principle it can be done with a computer program. Our results for the homotopy $7$-spheres are given in Section \ref{homotopy7sphsec} below where we have a table from which lower bounds for $|\pi_0(\gM^{SE}(\grS^7))|$ for each of the 28 homotopy $\grS^7$s can be easily obtained. For homotopy $S^9$s we have

\begin{theorem}\label{s9SEthm}
For the standard diffeomorphism type on $S^9$ we have $|\pi_0(\gM^{SE}(S^9))|\geq 983$ and for the exotic $\grS^9$ we have $|\pi_0(\gM^{SE}(\grS^9))|\geq 494$. Moreover, the components belong to distinct components of \linebreak $\gM^c_{+,0}(S^9)$ and $\gM^c_{+,0}(\grS^9)$, respectively.
\end{theorem}

The bounds obtained in this theorem as well as in dimension $7$ are probably far from sharp. In fact, for general dimensions it is not known whether $|\pi_0(\gM^{SE}(\grS^{2n+1}))|$ is finite or not. However, the methods used so far only give rise to finitely many components, albeit with double exponential growth with dimension (see Section \ref{Sylsec}).

It is known from Ustilovsky \cite{Ust99} that for all homotopy spheres $\grS^{4n+1}$ that bound a parallelizable manifold in dimension one mod four, we have $|\pi_0(\gM_{+,0}^c(\grS^{4n+1}))|=\aleph_0$. This result has been recovered by using symplectic homology in \cite{Gutt14,Gutt15} and one can use the formulae for symplectic homology in \cite{KwvKo13} to also obtain the result. However, none of these components are known to admit SE metrics with the exception of the BP link $L(2,2,2,3)$. See Remark \ref{SE2} below. That $|\pi_0(\gM_{+,0}^c(\grS^{4n+3}))|=\aleph_0$ was obtained in 9.5.10 in \cite{BG05} using the classical invariant of Morita \cite{Mor75,Sat77}, where $\grS^{4n+3}$ denotes a homotopy sphere in $bP_{4n+4}$.

Finally, in Section \ref{Sylsec} we show that the number of components of $\gM^{SE}(S^{4n+1})$ grows doubly exponentially with dimension.

\begin{ack}
This work was born at the AIM Workshop on Transversality in Contact Homology in Palo Alto, CA, Dec. 8-12, 2014 and we would like to thank the American Institute of Mathematics for its hospitality. We also thank Fr\'ed\'eric Bourgeois and Jean Gutt for interesting and helpful discussions. In addition the first author thanks Chi Li and Song Sun for discussions concerning Remark \ref{SE2} below.
\end{ack}

\section{Deformation Classes of Sasakian Structures}\label{defsassec}
First a warning concerning notational conventions. In Sasakian geometry $\xi$ is the Reeb vector field, and the contact bundle is usually denoted by $\cald$; whereas in contact topology $\xi$ denotes the contact bundle. Here to avoid confusion we do not use $\xi$ for either the Reeb field or the contact bundle. The Reeb field of a contact form $\eta$ can be $R_\eta$ or simply $R$ when there is no possibility of confusion. Note that the contact structure $\cald$ of any link of a weighted homogeneous polynomial (whp) must have $c_1(\cald)=0$. 

Recall (cf. Chapters 6 and 7 of \cite{BG05}) that a Sasakian structure is a contact metric structure with certain nice properties. A contact metric structure on an oriented manifold $M$ is a quadruple $(R,\eta,\Phi,g)$ where $\eta$ is a contact 1-form, $R$ its Reeb vector field, $\Phi$ a section of the endomorphism bundle ${\rm End}(M)$ satisfying $\Phi R=0,~\Phi\circ\Phi=-\BOne +R\otimes \eta$, and $g$ is a compatible Riemannian metric. Here compatible means 
$$g(\Phi X,\Phi Y)= g(X,Y) -\eta(X)\eta(Y)$$
holds for all vector fields $X,Y$. With contact structure $\cald=\ker\eta$ we see that $\Phi |_\cald=J$ is an almost complex structure on $\cald$. Thus, the pair $(\cald,J)$ is a strictly pseudoconvex almost CR structure on $M$ since $d\eta\circ (\Phi\otimes \BOne)$ gives the vector bundle $\cald$ a positive definite metric such that $g=d\eta\circ (\Phi\otimes \BOne)+\eta\otimes\eta$ is a Riemannian metric on $M$. The quadruple $\cals=(R,\eta,\Phi,g)$ is a {\it Sasakian structure} if $(\cald,J)$ is a CR structure, that is, $J$ is integrable, and $R$ is a Killing vector field, that is, $\pounds_Rg=0$. Note that this last condition is equivalent to $\pounds_R\Phi=0$, which means that $R$ is a infinitesimal CR transformation. Alternatively, a contact metric structure $\cals=(R,\eta,\Phi,g)$ is Sasakian if the metric cone $(C(M)=M\times \bbr^+,\bar{g}=dr^2+r^2g)$ is K\"ahler with K\"ahler form $d(r^2\eta)$. The complex structure $I$ on $C(M)$ satisfies 
$$IX=\Phi X+ \eta(X)\Psi, \quad  I\Psi=-R$$
where $X$ is a vector field on $M$ and $\Psi=r\d_r$ is the Liouville vector field.

\begin{definition}\label{consas}
 A contact structure $\cald$ is said to be of {\bf Sasaki type} if there exists a Sasakian structure $\cals=(R,\eta,\Phi,g)$ such that $\cald =\ker~\eta.$
\end{definition}

The Reeb vector field $R$ of a contact 1-form $\eta$ defines a 1-dimensional foliation $\calf_R$ and when the contact metric structure $\cals=(R,\eta,\Phi,g)$ is Sasakian, the transverse geometry of the foliation $\calf_R$ is K\"ahlerian; hence, $\calf_R$ is a Riemannian foliation. So a Sasakian structure fixes a  transverse holomorphic structure $\bar{J}$ on the normal bundle $\nu(\calf_R)$ to the foliation $\calf_R$ such that the diagram
\begin{equation}\label{Sasakianspace2eqn}
\begin{array}{ccccc}
TM & \stackrel{\Phi}{\ra{2.5}} &  TM & \\
  \decdnar{} \! \! \mbox{{\small $\pi_\nu$}}  & & \decdnar{} \! \! 
\mbox{{\small $\pi_\nu$}} & \\
   \nu(\calf_R) &\stackrel{\bar{J}}{\ra{2.5}} &\nu(\calf_R),&
\end{array}
\end{equation}
commutes.

Given a closed oriented smooth manifold $M$, we denote by $\gF(M)$ the space of all Sasakian structures on $M$ and give it the induced topology as sections of vector bundles. We denote by $\gF(M,R,\bar{J})$ the subset of Sasakian structures with Reeb vector field $R$, with transverse holomorphic structure $\bar{J}$, and the same complex structure on the cone $C(M)$. We give $\gF(M,R,\bar{J})$ the subspace topology. A Sasakian structure chooses a preferred basic cohomology class $[d\eta]_B\in H^{1,1}_B(\calf_R)$, and if $\cals'=(R,\eta',\Phi',g')$ is another Sasakian structures in $\gF(M,R,\bar{J})$, then we have $[d\eta']_B=[d\eta]_B$. So by the transverse $\d\bar{\d}$ lemma \cite{ElK} there is a basic function $\phi$, unique up to a constant, such that $d\eta'=d\eta +i\d_B\bar{\d}_B\phi.$ Furthermore, since the complex structure on the cone remains the same $\eta'=\eta+d_B^c\phi$ \cite{FOW06}, so we can identify $\gF(M,R,\bar{J})$ with the contractible space 
$$\{\phi\in C^\infty_B~|~(\eta+d_B^c\phi)\wedge (d\eta +i\d_B\bar{\d}_B\phi)^n\neq 0,~\int_M\phi~\eta\wedge (d\eta)^n=0\},$$
where $d_B^c=\frac{i}{2}(\bar{\d}-\d).$ For each such $\phi$ we have a Sasakian structure $\cals_\phi=(R,\eta_\phi,\Phi_\phi,g_\phi)$ with $\eta_\phi=\eta +d_B^c\phi,\Phi_\phi=\Phi-R\otimes d_B^c\phi,$ and $g_\phi=d\eta_\phi\circ (\Phi_\phi\otimes \BOne)+\eta_\phi\otimes\eta_\phi$ with the same Reeb vector field $R$, the same transverse holomorphic structure $\bar{J}$ and the same holomorphic structure on the cone. The class $[d\eta]_B\in H^{1,1}_B(\calf_R)$ is called a {\it transverse K\"ahler class} or {\it Sasaki class}. In the quasi-regular case, i.e.~the Reeb flow induces a locally free circle action, this class is also called a {\it Sasaki-Seifert structure} in \cite{BG05}.

This deforms the contact structure as well, sending $\cald\mapsto \cald_\phi=\ker\eta_\phi$. Of course, these give an isotopy of equivalent contact structures by Gray's Theorem. 
The space $\gF(M,R,\bar{J})$ described above is clearly infinite dimensional, so we want to factor this part out to get a finite dimensional moduli space. So we consider the identification space $\gF(M)/\gF(M,R,\bar{J})$. It is a pre-moduli space of Sasaki classes which we denote by $\gP\gM^c_\cals$. The Sasaki class of a Sasakian structure $\cals=(R,\eta,\Phi,g)$ is denoted by $\bar{\cals}$. Generally, $\gP\gM^c_\cals$ can be non-Hausdorff, but this is not the case with the examples treated here.

The diffeomorphism group $\dif$ acts on $\gF(M)$ by sending $\cals=(R,\eta,\Phi,g)$ to $\cals^\varphi=(\varphi_*R,(\varphi^{-1})^*\eta,\varphi^{-1}_*\Phi\varphi_*,(\varphi^{-1})^*g)$ for $\varphi\in \dif$. Then the transformed structure $\cals_\phi^\varphi$ belongs to the transformed space $\gF(M,\varphi_*R,\varphi^{-1}_*\bar{J}\varphi_*)$. So $\dif$ acts on $\gP\gM^c_\cals$ to define the {\it moduli space of Sasaki isotopy classes} $\gM^c_\cals(M)$. Generally very little is known about this space. Locally $\gM^c_\cals(M)$ is determined by the deformation theory of the transverse holomorphic structure of the foliation $\calf_R$. 

We recall the {\it type} of a Sasakian structure \cite{BG05}. 

\begin{definition}\label{Sastype}
A Sasakian structure $\cals=(R,\eta,\Phi,g)$ is {\bf positive (negative)} if the basic first Chern class $c_1(\calf_R)$ is represented by a positive (negative) definite $(1,1)$-form. It is {\bf null} if $c_1(\calf_R)=0$, and {\bf indefinite} if $c_1(\calf_R)$ is otherwise. 
\end{definition}

Here we are interested in positive Sasakian structures with $c_1(\cald)=0$. Generally, the type can change within a component of $\gM^c_\cals(M)$ as indicated in \cite{BoPa10} and is being studied further in \cite{BoTo15}; however, it follows from Corollary 7.5.26 of \cite{BG05} that type change cannot occur within a component when $c_1(\cald)=0$ or more generally a torsion class.

We wish to consider deformations of the Sasaki classes by deforming the transverse holomorphic structure where a good deformation theory exists (cf. Section 8.2.1 of \cite{BG05}). However, unlike K\"ahlerian geometry, such deformations generally do not remain Sasakian \cite{Noz14}. Fortunately, positive Sasakian structures are stable under small deformations (cf. Section 1.4 of \cite{Noz14}) and these are what interest us here. Accordingly, we denote the {\it moduli space of positive Sasaki classes with $c_1(\cald)=0$} by $\gM^c_{+,0}(M)$. We are also interested in the {\it moduli space of Sasaki-Einstein metrics} which is denoted by $\gM^{SE}(M)$. Generally there is very little known about the spaces $\gM^c_{+,0}(M)$ and $\gM^{SE}(M)$. Since Sasaki-Einstein structures are positive, there is a natural map $\gc:\gM^{SE}(M)\ra{1.6} \gM^c_{+,0}(M)$ which sends an SE structure to its Sasaki class. By a theorem of Nitta and Sekiya \cite{NiSe12} if $\gc(\cals)=\gc(\cals')$ then there is a $g$ in the connected component of the group of transverse holomorphic automorphisms such that $\cals'=g(\cals)$.

\section{Brieskorn Manifolds and Sasakian Geometry}
Here we give a brief review of weighted homogeneous hypersurfaces in $\bbc^{n+1}$ and their links. See Section 4.6 and  Chapter 9 of \cite{BG05} for full details. First we define a weighted $\bbc^*$ action on $\bbc^{n+1}$ by 
\begin{equation}\label{c*act}
\bfz=(z_0,\cdots,z_n)\mapsto (\grl^{w_0}z_0,\cdots,\grl^{w_n}z_n),
\end{equation}
where $\grl\in\bbc^*$ and which we denote by $\bbc^*(\bfw)$ where $\bfw=(w_0,\cdots,w_n)\in(\bbz^+)^{n+1}$ is the {\it weight vector} and the monomial $z_i$ is said to have weight $w_i$. This induces an action of $\bbc^*(\bfw)$ on the space of polynomials $\bbc[z_0,\ldots,z_n]$ whose eigenspaces are {\it weighted homogeneous polynomials} of {\it degree} $d$ defined by 
\begin{equation}\label{whp}
f(\lambda^{w_0} z_0,\ldots,\lambda^{w_n}
z_n)=\lambda^df(z_0,\ldots,z_n)\, .
\end{equation}
We want to make two assumptions about $f$. First, $f$ should have only an isolated singularity at the origin in $\bbc^{n+1}$, and second, $f$ does not contain monomials of the form $z_i$ for any $i=0,\dots,n$. This last condition eliminates the standard constant curvature metric sphere.

The zero locus of $f$ is referred to as a hypersurface singularity and denoted by $C_f$.
The link is defined by intersecting the zero locus of $f$ with unit sphere in $\bbc^{n+1}$, namely
\begin{equation}\label{link}
L^*(\bfw,d)=\{f(\bfz)=0\}\cap S^{2n+1}.
\end{equation}
It follows from the Milnor Fibration Theorem that $L(\bfa)$ is $(n-2)$-connected. 
It was realized in the mid 70's that links of weighted homogeneous polynomials are a good source of contact manifolds \cite{AbEr75a,LuMe76,SasHsu76}. Furthermore, they have a natural Sasakian structure \cite{Abe77,Tak78}. Hence,

\begin{theorem}\label{whpsas}
Links of weighted homogeneous polynomials admit contact structures of Sasaki type.
\end{theorem}

In fact a link $L^*(\bfw,d)$ has a natural Sasakian structure $\cals_\bfw=(R_\bfw,\eta_\bfw,\Phi_\bfw,g_\bfw)$ which is simply the restriction of the weighted Sasakian on the sphere $S^{2n+1}$ with Reeb vector field $R_\bfw=\sum_jw_jH_j$ where $H_j=-i(z_j\d_{\color{red}z_j}-\bar{z}_j\d_{\bar{z}_j})$.

Of special interest to us are the Brieskorn-Pham polynomials which are polynomials of the form 
\begin{equation}\label{BPpoly}
f(\bfz)=z_0^{a_0}+\cdots +z_n^{a_n}.
\end{equation}
Here the exponent vector $\bfa=(a_0,a_1,\ldots,a_n)\in (\bbz^+)^{n+1}$ is related to the weights by $a_iw_i=d$ for all $i=0\cdots,n$. In this case we write the link as 
$$L(\bfa)=\{f(\bfz)=0\}\cap S^{2n+1}$$
and refer to it as a Brieskorn manifold. We also require that $a_j\geq 2$ since a linear term in the polynomial gives rise to the standard sphere with its standard contact structure.

Note that the quotient of $L^*(\bfw,d)$ by the weighted circle action $(z_0,\cdots,z_n)\mapsto (\grl^{w_0}z_0,\cdots,\grl^{w_n}z_n)$ with $|\grl|=1$ is a projective algebraic orbifold $\calz_\bfw$. Equivalently, $\calz_\bfw=(C_f\setminus \{0\})/\bbc^*(\bfw)$.
We denote by $\Theta_\calz$ the sheaf of germs of holomorphic vector fields on $\calz$.

\begin{lemma}\label{poslem}
Let $L^*(\bfw,d)$ be the link of a weighted hypersurface singularity of degree $d$ with $n\geq 3$, weight vector $\bfw$, and with an isolated singularity at the origin. Let $\cals_\bfw$ be the naturally induced Sasakian structure on $L^*(\bfw,d)$. 
\begin{enumerate}
\item Then any deformation of the transverse holomorphic structure of $\cals_\bfw$ remains Sasakian. 
\item Moreover, if $\cals_\bfw$ is positive, the deformation remains positive.
\item If $w_i<\frac{d}{2}$ for all but one of the weights and $H^2(\calz_{\bfw,},\Theta_{\calz_\bfw})=0$, then the component of $\gM^c_{+,0}(L^*(\bfw,d))$ containing $\bar{\cals}_\bfw$ is locally modeled on the complex vector space $H^1(\calz_{\bfw,},\Theta_{\calz_\bfw})$.
\end{enumerate}
\end{lemma}

\begin{proof}
Since a link $L^*(\bfw,d)$ with $n\geq 3$ is simply connected, (1) follows from \cite{Noz14}. For (2) Proposition 9.2.4 of \cite{BG05} implies that $c_1(\cald_\bfw)=0$ and then Corollary 7.5.26 of \cite{BG05} implies that type change cannot occur which implies (2). To prove (3) we note that Proposition 37 and Remark 38 of \cite{BGK05} implies that the group of automorphisms is discrete\footnote{Contrary to the statement in \cite{BGK05}, there is a case, namely when $n=3$ and $|\bfw|=d$, that the automorphism group is not a finite group, but an infinite discrete group. See Lemma 5.5.3 of \cite{BG05} and the corresponding references there.}, hence, $H^0(\calz,\Theta_\calz)=0$, so the Kuranishi space of deformations of transverse holomorphic foliations is locally modeled on $H^1(\calz_{\bfw,},\Theta_{\calz_\bfw})$ by Proposition 8.2.6 of \cite{BG05}. 
\end{proof}

Let $\gM^{WH,c}_{+,0}(M)$ denote the subset of $\gM^c_{+,0}(M)$ consisting of Sasaki classes that can be represented as the link of a weighted homogeneous hypersurface. It is well known that a Sasaki class that can be represented as a weighted homogeneous hypersurface (or more generally a complete intersection) must have $c_1(\cald)=0$ and whose type is either positive, negative, or null. It is also known that being represented by a weighted homogeneous hypersurfaces places topological restrictions on $M$. For example, such an $(2n+1)$-manifold must be $n-1$ connected.

Here we are interested only in those of positive type. Similarly $\gM^{WHSE}(M)$ denotes the subset of $\gM^{SE}(M)$ that can be represented by a weighted homogeneous hypersurface. However, we need to exclude the standard contact structure $\cald_0$ on the sphere $S^{2n+1}$. So when $M=S^{2n+1}$ an element of $\gM^{WHSE}(M)$ corresponds to an exotic contact structure.

The problem of giving sufficient numerical conditions for the existence of Sasaki-Einstein (K\"ahler-Einstein orbifold) metrics has been well studied. Generally there are no known numerical conditions which are both necessary and sufficient. Here we give sufficient conditions for Brieskorn manifolds.

\begin{theorem}\label{SEcond}
Let $L(\bfa)$ be the link of a Brieskorn hypersurface. Then if either one of the following two conditions hold $L(\bfa)$ admits a Sasaki-Einstein metric:
\begin{eqnarray*}
1< \sum_{i=0}^n\frac{1}{a_i} & <  & 1+\frac{n}{n-1}{\rm min}_{i,j}\bigg\{\frac{1}{a_i},\frac{1}{b_ib_j}\bigg\} \\
1< \sum_{i=0}^n\frac{1}{a_i} & <  & 1+\frac{n}{n-1}{\rm min}_i\bigg(\frac{1}{a_i}\bigg){\rm max}_j\bigg(\frac{1}{a_j}\bigg),
\end{eqnarray*}
where $b_i=\gcd(\lcm_{j\neq i}a_j,a_i)$
\end{theorem}

\begin{proof}
The proof of the first condition is given in \cite{BGK05} while that of the second is given for $n=3$ in \cite{BN09}. The extension of the latter to arbitrary $n$ is straightforward.
\end{proof}

In the special case when the components of $\bfa$ are pairwise relatively prime Ghigi and Koll\'ar \cite{GhKo05} have given a better estimate. Indeed owing to the Lichnerowicz bound obtained by Gauntlett, Martelli, Sparks, and Yau \cite{GMSY06} this bound is sharp.

\begin{theorem}\label{GhKobound}
Let $L(\bfa)$ be the link of a Brieskorn hypersurface. Assume further that the components $a_0,\cdots,a_n$ of $\bfa$ are pairwise relatively prime. Then $L(\bfa)$ admits a Sasaki-Einstein metric if and only if 
$$1< \sum_{i=0}^n\frac{1}{a_i} <  1+n~{\rm min}_{i}\bigg\{\frac{1}{a_i}\bigg\}.$$
\end{theorem}

For any link of a weighted homogeneous hypersurface, the result of \cite{GMSY06} is: 

\begin{theorem}\label{Lichbd}
Let $L^*(\bfw,d)$ be a smooth link of a weighted homogeneous hypersurface. If the inequality $|\bfw|-d> n~{\rm min}_iw_i$ holds, then $L^*(\bfw,d)$ cannot admit a Sasaki-Einstein metric.
\end{theorem}

\subsection{Perturbations of Brieskorn Manifolds}
In order to study moduli we consider perturbations of the Brieskorn polynomials by a weighted homogeneous polynomial $p$ of degree $d$. We now consider the hypersurface singularity $C_{f+p}$ and require that the intersection of its zero locus with any number of hyperplanes $(z_i=0)$ is smooth outside of the origin in $\bbc^{n+1}$. With this the corresponding links $L_f$ and $L_{f+p}$ are diffeomorphic and the corresponding contact structures are equivalent by Gray's Theorem. The polynomial $p$ consists of monomials of the form $z_0^{b_0}\cdots z_n^{b_n}$ where $0\leq b_j<a_j$ and $\sum_jb_jw_j=d$. Now let us consider the projective complex orbifolds, $C_f/\bbc^*(\bfw)$ and $C_{f+p}/\bbc^*(\bfw)$. By the Grothendieck-Lefschetz Theorem \cite{sga2}, it follows that any isomorphism of $C_f/\bbc^*(\bfw)$ and $C_{f+p}/\bbc^*(\bfw)$ is induced from an isomorphism of $\bbc^{n+1}$ that commutes with $\bbc^*(\bfw)$. As discussed in \cite{BGK05} such elements form a group, denoted by $\gA\gu\gt(\bbc^{n+1},\calo(\bfw))$ with dimension $\sum_ih^0(\bbc\bbp(\bfw),\calo(w_i))$. Here $h^0(\bbc\bbp(\bfw),\calo(m))$ denotes the complex dimension of the complex vector space $H^0(\bbc\bbp(\bfw),\calo(m))$ of holomorphic sections of the orbibundle $\calo(m)$. On the other hand the added polynomial gives a space of dimension $h^0(\bbc\bbp(\bfw),\calo(d))$, so we have arrived at

\begin{proposition}\label{hyslocmod}
Let $M$ be a smooth manifold represented by the link $L(\bfa)$ of a Brieskorn hypersurface that admits a Sasaki-Einstein metric. Assume further that the components of the exponent vector $\bfa$ has at most one $2$. Then near $L(\bfa)$ the moduli space $\gM^{WHSE}(M)$ has complex dimension 
$$h^0(\bbc\bbp(\bfw),\calo(d))-\sum_ih^0(\bbc\bbp(\bfw),\calo(w_i)).$$
\end{proposition}

\begin{remark}\label{SE2}
It is interesting to note that until recently all known examples of Brieskorn manifolds $L(\bfa)$ that admit a Sasaki-Einstein metric with the exception of a quadric have at most one exponent equal to $2$. However, there was a controversy concerning the link $L(2,2,2,3)$. This does not appear on the list of cohomogeneity one Sasaki-Einstein manifolds given in \cite{Con07} implying that it does not admit an SE metric. On the other hand in a recent paper Li and Sun \cite{LiSu14} claim that it does. This controversy has been recently resolved by Chi Li \cite{Li15} in favor of the Li-Sun result.  It can be noted that the BP links $L(2,2,2,k)$ for $k>3$ cannot admit SE metrics as they are obstructed by the Lichnerowicz obstruction Theorem \ref{Lichbd}. Nevertheless, there are an infinite number of links of the form $L(2,2,p,q)$ where the existence or non-existence of an SE metric is unknown.
\end{remark}

\begin{example}\label{S5SEex}
Applying Proposition \ref{hyslocmod} to $M=S^5$ the results of \cite{BGK05} and \cite{GhKo05} show that there are at least $80$ components of $\gM^{WHSE}(S^5)$ (excluding the round sphere SE metric) plus the Li-Sun case of Remark \ref{SE2} which gives $81$. However, there is an example of a family of SE metrics on $S^5$ that was not included in \cite{BGK05} since it does not satisfy the first estimate in Theorem \ref{SEcond}. This is given in Theorem \ref{newSE} below. So we know that $\gM^{WHSE}(S^5)$ has at least $82$ components. There are $55$ components consisting of single points, eighteen components of real dimension $2$, four components of real dimension $4$, one component of real dimension $6$, two components of real dimension $8$, one component of real dimension $10$, and one component of real dimension $20$. This last family arises from

\begin{theorem}\label{newSE}
The link $L(2,3,11,11)\approx S^5$ admits a $20$ parameter family of Sasaki-Einstein metrics.
\end{theorem} 

\begin{proof}
One easily checks that the second estimate in Theorem \ref{SEcond} is satisfied. Moreover, from Proposition \ref{hyslocmod} one sees that the local moduli space has real dimension $20$.
\end{proof}

This provides the largest known component of $\gM^{SE}(S^5)$.

\end{example}

\begin{remark}
Notice that if the components of $\bfa$ are pairwise relatively prime, there is no polynomial $p$ to add, so the moduli is a single point. Also if there are two isolated points in the Brieskorn graph and only a quadric can be added, the moduli is a single point.

\end{remark}

\section{$S^1$-Equivariant Symplectic Homology and Brieskorn Manifolds}
In order to distinguish components of $\gM^c_{+,0}(M)$ we make use of the fact that the $+$ component of $S^1$-equivariant symplectic homology of Liouville fillings of Brieskorn manifolds is a contact invariant under some mild and checkable conditions, made precise in Lemma~\ref{lemma:equi_SH_invariant}.
Furthermore, the mean Euler characteristic of this homology, defined in~\eqref{eq:def_mec}, is always an invariant of the natural contact structure on a Brieskorn manifold.

\subsection{On symplectic fillings}
Many contact manifolds arise naturally as the boundary of a symplectic or complex manifold.
The notions that are relevant for this paper are the following.
\begin{definition}
Suppose that $(M,{\mathcal D})$ is a compact, coorientable contact manifold.
A {\bf strong filling} or {\bf convex filling} of $(M,{\mathcal D})$ consists of a symplectic manifold $(W,\omega)$ such that 
\begin{itemize}
\item the boundary of $W$ is diffeomorphic to $M$
\item there is a Liouville vector field $X$ (i.e.~$\mathcal L_X \omega=\omega$) defined in a neighborhood of the boundary which is pointing outward.
\item the kernel of the $1$-form $(i_X \omega)|_M$ is the contact structure ${\mathcal D}$. 
\end{itemize}
We will refer to the boundary of a strong symplectic filling $(W,\omega)$ as the {\bf contact type boundary}.

A {\bf Liouville filling} of $(M,{\mathcal D})$  is a strong filling of $(M,{\mathcal D})$ for which the Liouville vector field $X$ is globally defined.
\end{definition}
These are purely symplectic notions.
For fillings related to complex geometry, we first define $J$-convexity.
\begin{definition}
Assume that $M$ is a real hypersurface in a complex manifold $(W,J)$.
We obtain a distribution ${ \mathcal D}$ on $M$ by setting $\mathcal D := TM \cap JTM$.
We say that $M$ is a $J$-convex hypersurface if there exists a $1$-form $\eta$ on $M$ such that $\mathcal D =\ker \eta$ and $-\eta([X,JX])>0$ for all non-zero vectors $X\in \mathcal D$.

A {\bf holomorphic filling} of a compact, coorientable contact manifold $(M,{\mathcal D})$ consists of a compact, complex manifold $(W,J)$ such that
\begin{itemize}
\item the boundary of $W$ is diffeomorphic to $M$.
\item the boundary of $W$ is $J$-convex and $\mathcal D=TM \cap JTM$.
\end{itemize}

A {\bf Stein filling} of a contact manifold $(M,{\mathcal D})$ is a holomorphic filling of $(M,{\mathcal D})$ that is biholomorphic to a Stein domain.
\end{definition}
Note that a Liouville filling for $(M,{\mathcal D})$ is an exact symplectic manifold $(W,\omega=d\lambda )$, where the Liouville $1$-form is defined by $i_X \omega=:\lambda$.
We will write $(W,\lambda)$ for such an exact symplectic manifold with contact type boundary and call it a Liouville domain.

We claim that a Stein filling $(W,J)$ is automatically Liouville.
To see this, choose a strictly plurisubharmonic function $f:W\to \bbr$ such that the boundary of $W$ is a regular level set. We get the $1$-form $\lambda=-d f\circ J$, and the K\"ahler form $d\lambda$.
The vector field $X$ satisfying $i_X\omega=\lambda$ is hence Liouville and it is transverse to regular level sets of $f$.

\begin{theorem}
Every Sasakian manifold is holomorphically fillable.
\end{theorem}

In fact, since a Sasakian structure has an underlying strictly pseudoconvex CR structure, it follows from a theorem of Rossi, cf. Theorem 5.60 of \cite{CiEl12}, that a contact structure of Sasaki type is holomorphically fillable if the dimension of $M$ is at least $5$, and hence, K\"ahler fillable by Theorem 5.59 of \cite{CiEl12}.
The $3$-dimensional case was later also shown to be holomorphically fillable by Marinescu and Yeganefar, \cite{MaYe07}. Due to results of Bogomolov and de Oliveira, \cite{BdO96}, the $3$-dimensional case then turns out to be Stein fillable.

\begin{remark}
Not every Sasakian manifold is Stein fillable.
A well-known example is $(\bbr \bbp^{2n+1},\bar {{\mathcal D}}_0)$ with its induced contact structure from the standard contact sphere $(S^{2n+1},{{\mathcal D}}_0)$ for $n>1$.
Its cohomology ring obstructs the existence of Stein filling, see \cite[Example 5.62]{CiEl12}.
It is an open question whether $(\bbr \bbp^{2n+1},\bar {{\mathcal D}}_0)$ admits a Liouville filling.
\end{remark}

For the invariants discussed in this paper, the existence of a Liouville filling is required.
By smoothing the singularity we get the following result.
\begin{proposition}
A Brieskorn-Pham link is Stein fillable.
\end{proposition}

\subsection{Equivariant symplectic homology}
In this section we survey the work of Bourgeois and Oancea, \cite{BO:S1_Fredholm,BoOa13a}. We adapt Abouzaid's description of orientations, \cite[Chapter 1.4]{Abo13}, because it turns out to be useful for the Morse-Bott point of view.

We first give a very brief summary of the construction. Equivariant symplectic homology mimics the Borel construction for $S^1$ and applies it to the Floer homology of a functional on the loop space of a symplectic manifold: in particular, generators of the chain complex are (circles of) critical points, and the differential counts $0$-dimensional families of the moduli space of solutions to a PDE describing the ``gradient flow''.
The construction depends on the chosen functional. To remove this dependence, one considers a class of functionals and takes a direct limit over the corresponding homologies.

We now elaborate a bit. To keep things as simple as possible we will outline the construction for a Liouville domain $(W,\lambda)$ satisfying the following condition
\begin{enumerate}
\item[(CF)] the Liouville domain $W$ is simply-connected and has torsion first Chern class, so $[c_1(W)]_{\bbr}=0$ as an element of $H^2_{dR}(W)$.
\end{enumerate}
Denote the boundary by $M=\partial W$ and the contact form on the boundary by $\eta:=\lambda|_{M}$.
The contact structure will be denoted by $\mathcal D$.
Note that for any Liouville filling we have $TW|_{M=\partial W}\cong \langle X,R_\eta \rangle \oplus {\mathcal D}$, where $X$ denotes the Liouville vector field.
Hence $c_1({\mathcal D}=\ker \eta )=i^*c_1(TW)$ is torsion, where $i: M \to W$ is the inclusion.
We will refer to such a Liouville domain as a {\it convenient filling}.

Define the complete Liouville manifold $\bar W$ by attaching the positive end of a symplectization,
\begin{equation}
\label{eq:complete_Liouville}
\bar W=W \cup_\partial ( [1,\infty[ \times M,d(r\lambda|_M)\, ),
\end{equation}
where $r$ is the coordinate in $[1,\infty[$.
Call the extended Liouville form $\bar \lambda$.
Introduce a free circle action on $S^1\times \bar W \times S^{2N+1}$ given by
\begin{equation}
\label{eq:circle_action}
\theta\cdot(t,x,z)=(t+\theta,x,z\cdot \theta).
\end{equation}
Choose a function $H:S^1\times \bar W \times S^{2N+1} \to \bbr$ satisfying
\begin{itemize}
\item $H$ is invariant under the circle action~\eqref{eq:circle_action}.
\item $H|_{S^1\times[r_0,\infty[\times S^{2N+1} }=Sr+b(z)$, for some $r_0 \geq 1$.
Here $S$ is a positive real number that is not equal to the period of any periodic Reeb orbit of $\eta$.
\end{itemize}
We will call $H$ a parametrized time-dependent Hamiltonian.
Denote the free loop space of $\bar W$ by $\Lambda \bar W$ and define the parametrized action functional by
\[
\begin{split}
\mathcal A^N:\Lambda \bar W \times S^{2N+1} & \longrightarrow \bbr \\
(x,z) & \longmapsto -\int_{S^1} x^* \bar \lambda
-\int_0^1 H(t,x(t),z)dt.
\end{split}
\]
We describe the ``Morse'' homology of this action functional using Floer theory following the work of Bourgeois and Oancea.
Due to the $S^1$-invariance of $H$ and hence $\mathcal A^N$, critical points always come in an $S^1$-orbit.
The set of all critical points is given by
$$
\mathcal P(H)=\{(x,z)\in \Lambda \bar W\times S^{2N+1}~|~\dot x=X_H(x),\ \int_0^1 \vec \nabla_z H(t,x(t),z)dt=0 
\}
.
$$
Here $\vec \nabla_z$ denotes the gradient with respect to an auxiliary $S^1$-invariant metric $g$ on $S^{2N+1}$.

We choose additional auxiliary data, namely a family of almost complex structures $J:S^1\times \bar W \times S^{2N+1}\to End(T\bar W)$ satisfying 
\begin{itemize}
\item $J(\cdot,x,\cdot)$ is a compatible complex structure for $(T_x\bar W,d_x\bar \lambda)$.
\item for $r_0$ sufficiently large $J|_{S^1\times [r_0,\infty[ \times M \times S^{2N+1}}$ satisfies $J r\partial_r=R_\eta$ and $\mathcal L_{r\partial_r} J=0$. 
\item $J$ is $S^1$-invariant in the sense that $J(t+\theta,x,z\cdot \theta)=J(t,x,z)$ for all $\theta\in S^1$.
\end{itemize}
This gives rise to a family of Riemannian metrics $G_{t,z}(x)(v,w)=\omega(v,J(t,x,z)w )$ on $\bar W$, which in turn yields a family of $L^2$-metrics.

The formal negative $L^2$-gradient flow of $\mathcal A^N$ motivates the \linebreak parametrized Floer equations.
To define these, we give the cylinder $Z=\bbr\times S^1$ coordinates $s,t$, where $S^1=\bbr/\bbz$.
The parametrized Floer equations are then 
\begin{equation}
\label{eq:param_floer}
\begin{split}
\bar u=(u,z):Z & \longrightarrow \bar W\times S^{2N+1} \\
\frac{\partial}{\partial s} u+J(t,u,z)(\frac{\partial u}{\partial t}-X_H) &=0\\
\frac{d}{ds}z-\int_{0}^1 \vec \nabla_z H(t,u(s,t),z(s))\,dt &=0\\
\lim_{s\to \mp \infty} \bar u(s,t)&\in S_{\pm}
\end{split}
\end{equation}
where $S_{\pm}$ are $S^1$-orbits of critical points of $\mathcal A^N$.
We will denote the moduli space of parametrized Floer trajectories by
$$
\mathcal M(S_+,S_-;H,J,g)=\{ \bar u \text{ solves \eqref{eq:param_floer}} \}/\bbr,
$$
where the $\bbr$-action is induced by reparametrizations in the $s$-direction.
Note that $\mathcal M(S_+,S_-;H,J,g)$ still carries a free circle action.
By Theorem~B in \cite{BO:S1_Fredholm} the quotient space is a smooth manifold for generic Floer data.
\begin{theorem}
For a generic choice of Floer data $(H,J,g)$ the moduli space $\mathcal M_{S^1}(S_+,S_-;H,J,g)=\mathcal M(S_+,S_-;H,J,g)/S^1$ is a smooth manifold of dimension
$$
-\mu(S_+)+\mu(S_-)-1.
$$
\end{theorem}
Here $\mu(S)$ denotes the parametrized Robbin-Salamon index of $S$, which is explained below in Section \ref{sec:grading}.
\subsubsection{Capping disks and trivializations}
We will now associate a graded abelian group of rank $1$, so a group isomorphic to $\bbz$, with each $S^1$-orbit of critical points.
To do so, note that given an $S^1$-orbit of critical points, say $S=S^1\cdot (x_S,z_S)$, we obtain an $S^1$-family of solutions to the parametrized Floer equations, generated by
$$
u_S(s,t)=(x_S(t),z_S).
$$
Linearize the parametrized Floer equations along this circle of solutions to obtain an $S^1$-family of Fredholm operators, namely for $\theta\in S^1$,
$$
D_{\theta\cdot \bar u_S}:W^{1,p}_\delta(Z,\theta\cdot \bar u_S^*(T\bar W\oplus TS^{2N+1})) \longrightarrow
L^{p}_\delta(Z,\theta\cdot \bar u_S^*(T\bar W\oplus TS^{2N+1})).
$$

An expression for this family of operators and a proof that this family consists of Fredholm operators is given in \cite{BO:S1_Fredholm}.

For each circle of critical points, $S=S^1\cdot(x_S,z_S)$, choose 
\begin{itemize}
\item capping disks $d_S=(d_{S,W},d_{S,S^{2N+1}}):D^2\to  \bar W\times S^{2N+1}$, i.e.~curves connecting a constant loop $\{x_0\} \times\{ z_0 \} \in \Lambda\bar W\times S^{2N+1}$ to the loop $\theta\cdot (x_S,z_S)$;
\item a symplectic trivialization $\epsilon_{S,W}:D^2\times (\bbc^n,\omega_0 )\to d_{S,W}^*T\bar W$;
\item a trivialization $\epsilon_{S,S^{2N+1}}:D^2\times \bbr^{2N+1}\to d_{S,S^{2N+1}}^*T S^{2N+1}$.
\end{itemize}
With respect to these trivializations we get an $S^1$-family of operators
$$
D_{S}: W^{1,p}_\delta(\bbc,d_S^*(T\bar W\oplus TS^{2N+1})) \longrightarrow
L^{p}_\delta(\bbc,d_S^*(T\bar W\oplus TS^{2N+1})).
$$
\subsubsection{Graded abelian groups}
\label{sec:grading}
We associate a determinant line bundle with this family of operators by
$$
Det(D_S)=\Lambda^{top} \ker D_S \otimes \Lambda^{top} \coker D_S^*.
$$
This is a real line bundle over the circle $S$.
\begin{lemma}
Suppose that $S$ is an $S^1$-orbit in $\mathcal P$.
Then the determinant bundle $Det \, D_S\to S$ is a trivial line bundle.
\end{lemma}
By the lemma, the determinant bundle is trivial along any $S^1$-orbit $S$ in $\mathcal P(H)$.
We have hence two orientations, i.e.~homotopy classes of non-vanishing sections, which we call $\delta_S^+$ and $\delta_S^-$.
Define the orientation line of $S$ as the free abelian group of rank $1$, given by
$$
o_{S}:=\langle \delta_{S}^+,\delta_{S}^-~|~\delta_{S}^+ +\delta_{S}^-=0 \rangle
.
$$
We now come to the grading.
To define the parametrized Robbin-Salamon index $\mu(S)$ of the $S^1$-orbit $S^1\cdot(x_S,z_S)$, consider the extended Hamiltonian $\tilde H:S^1 \times (\bar W,d\bar \lambda) \times (T^*S^{2N+1},dz\w dp) \to \bbr$, defined by $\tilde H(t,x;z,p)=H(t,x,z)$.
By the assumption that $(x_S,z_S)$ is a critical point, the map
$$
t\longmapsto \tilde x=(x_S(t),z_S,p(t)=p(0)-\int_0^t \vec \nabla_z H(\tau,x_S(\tau),z_S)d\tau )
$$
gives a $1$-periodic orbit of $\tilde H$.
The index $\mu(S)$ is then defined as the usual Robbin-Salamon index of the $1$-periodic orbit $\tilde x$ with respect to the symplectic trivialization induced by $\epsilon_{S,W}$ and $\epsilon_{S,S^{2N+1}}$.

We will grade the orientation line $o_S$ by $-\mu(S)+N$. 
The choice of the shift by $N$ is used in Theorem~\ref{thm:continuation}.
We now define the Floer chain complex as the graded $\bbz$-module freely generated by the orientation lines of $\mathcal P(H)$,
$$
SC_*^{S^1,N}(W,H,J,g)=\bigoplus_{S \in \mathcal P (H)}  o_S
.
$$
Given a regular parametrized Floer trajectory $u\in \mathcal M(S_+,S_-;H,J,g)$ of index $1$, we construct an isomorphism map $\partial_u: o_{S_+}\longrightarrow o_{S_-}$ by a gluing construction, see Lemma~1.5.4 from~\cite{Abo13}.
\begin{lemma}
Suppose that $u\in \mathcal M(S_+,S_-;H,J,g)$ is regular  of index $1$. Then $\partial_u$ is an isomorphism.
\end{lemma}
We define the restriction of the differential of the equivariant Floer complex to an orientation line by
$$
\partial^{S^1}|_{o_{S_+}}:=\sum_{\underset{-\mu(S_+)+\mu(S_-)=1}{S_-\in \mathcal P}} \sum_{[u]\in \mathcal M_{S^1}(S_+,S_-;H,J,g)} \partial_u.
$$
\begin{remark}
By choosing \emph{coherent} (i.e.~gluing compatible) orientations one can identify each orientation line with $\bbz$.
The isomorphisms $\partial_u$ then simply map $1\mapsto \epsilon(u) \cdot 1$, where the sign $\epsilon(u)$ is determined by comparing the generator of the $S^1$-action, which lies in the kernel of $D_u$, with the coherent orientation on the moduli space.
\end{remark}

We have the following \cite[Proposition 4.5]{BoOa13a}
\begin{proposition}
The map $\partial^{S^1}$ satisfies ${\partial^{S^1}}\circ {\partial^{S^1}}=0$. 
\end{proposition}
Equivariant Floer homology is then defined as
$$
SH^{S^1,N}(W,H,J):=H_*(SC_*^{S^1,N}(W,H,J),\partial^{S^1}).
$$
This homology group depends on the choice of Hamiltonian and on $N$.
To remove this dependence one has the following theorem, \cite[Lemma 5.6]{BoOa13a}.
\begin{theorem}
\label{thm:continuation}
Assume that the slope of $H_2$ is greater than or equal to that of $H_1$, and assume that $N_2\geq N_1$.
Then there are chain maps
$$
c_{12}:SC_*^{S^1,N_1}(W,H_1,J_1,g_1) \longrightarrow SC_*^{S^1,N_2}(W,H_2,J_2,g_2).
$$
\end{theorem}
We will refer to these chain maps as continuation maps: they are defined by counting rigid parametrized Floer trajectories under an increasing homotopy  of Hamiltonians.
Equivariant symplectic homology is then defined as
$$
SH^{S^1}(W):=\varinjlim_{N} \varinjlim_{S}SH^{S^1,N}(W,H_{S,N},J_{S,N},g_{S,N}).
$$
This homology no longer depends on the choice of complex structure and Hamiltonian.

From now on, we suppress the $(W,H,J,g)$-dependence in the notation of the chain complexes to shorten notation.
We now explain that equivariant symplectic homology comes equipped with a tautological exact sequence.
To define this, fix $\epsilon>0$ small, and choose a parametrized Hamiltonian with the following properties:
\begin{itemize}
\item  $H(t,\cdot,z)$ has $C^2$-norm less than $\epsilon$ on $W$ for every $(t,z) \in S^1\times S^{2N+1}$.
\item $H$ is positive on the whole cylindrical end $S^1\times [1,\infty[ \times M \times S^{2N+1}$.
\item $H$ has slope $S$ on the cylindrical end $S^1\times [r_0,\infty[ \times M \times S^{2N+1}$, for some $r_0 \geq 1$.
\end{itemize}
One chooses $\epsilon$ so small such that all critical points of $\mathcal A^N$ with action less than $\epsilon$ are fixed points.
Since the differential $\partial^{S^1}:SC^{S^1,N}_*\to SC^{S^1,N}_{*-1}$ is action decreasing, we obtain a subcomplex,
$$
(SC^{-,S^1,N}_*,\partial^{S^1})=(\{ S_{(x,z)}\in SC^{S^1,N}_*~|~\mathcal A^{N}(x,z)<\epsilon \},\partial^{S^1}).
$$
Denote the quotient complex by $SC^{+,S^1,N}=SC^{S^1,N}/SC^{-,S^1,N}$, so we have a short exact sequence of chain complexes,
$$
0\longrightarrow SC^{-,S^1,N}_*
\longrightarrow SC^{S^1,N}_*
\longrightarrow SC^{+,S^1,N}_*
\longrightarrow 0.
$$
If we denote the associated homologies by $SH$, we get a long exact sequence in homology,
$$
\longrightarrow
SH^{-,S^1,N}_*
\longrightarrow SH^{S^1,N}_*
\longrightarrow SH^{+,S^1,N}_*
\longrightarrow
SH^{-,S^1,N}_{*-1}
\longrightarrow
.
$$
Now take the direct limit over the slopes and over the inclusion $S^{2N+1}\to S^{2N'+1}$ to obtain $SH^{-,S^1}(W)$ and $SH^{+,S^1}(W)$.
These homology groups do not longer depend on the the choice of complex structure and Hamiltonian.
The above tautological exact sequence then becomes the so-called Viterbo sequence
\begin{theorem}
There is a long exact sequence
$$
\longrightarrow
SH^{-,S^1}_*(W)
\longrightarrow SH^{S^1}_*(W)
\longrightarrow SH^{+,S^1}_*(W)
\longrightarrow
SH^{-,S^1}_{*-1}(W)
\longrightarrow
.
$$
\end{theorem}
The homology groups $SH^{-,S^1}$ measure the differential topology of $W$ in the following sense, see \cite{BoOa13a}, Theorem 4.7.
\begin{theorem}
\label{thm:SH-}
There is an isomorphism
$$
SH^{-,S^1}_*(W)\cong H_{*+n}(W,\partial W;\bbz)\otimes H_*(\bbc P^\infty;\bbz).
$$
\end{theorem}

\begin{remark}
Morally we may think of $SC^{+,S^1,N}$ as being generated by $1$-periodic orbits that are not fixed points, and by a judicious choice of Hamiltonians we can think of $SH^{+,S^1}$ as being generated by periodic Reeb orbits on the boundary.
This picture is complicated by the fact that the differential counts Floer trajectories which may go through the filling.
The interpretation in terms of the Reeb flow can be made more precise with the Morse-Bott spectral sequence, see Theorem~\ref{thm:SS_equivariant_homology} for a formulation in a special case.
\end{remark}

Concerning the symplectic invariance of equivariant symplectic homology, the following theorem, see \cite{Gutt15}, will be relevant.
\begin{theorem}
\label{thm:SH_S1_sympl_def_invar}
If $W$ is convenient filling, then $SH^{+,S^1}(W)$ is independent of the choice of Liouville form $\lambda$.
\end{theorem}

\subsection{Mean Euler characteristic}
From now on, we will always use $\bbq$-coefficients for equivariant symplectic homology.
Denote the Betti numbers of the $+$-part of equivariant symplectic homology of a convenient Liouville domain $(W,\lambda)$ by $sb_i:=\rk SH^{+,S^1}_i(W)$.
We define the {\bf mean Euler characteristic} of $(W,\lambda)$ as
\begin{equation}
\label{eq:def_mec}
\chi_m(W) = 
\frac{1}{2} 
\left( 
\liminf_{N \to \infty}  \frac{1}{N} \sum_{i=-N}^N (-1)^i sb_i(W) \,+\,
\limsup_{N \to \infty}  \frac{1}{N} \sum_{i=-N}^N (-1)^i sb_i(W)
\right) 
\end{equation}
if this number exists.
This is the case if there is uniform bound on the Betti numbers $sb_i$.

\subsection{Autonomous Hamiltonians, spectral sequence and contact invariants}
To compute equivariant symplectic homology in practice it is useful to work with autonomous Hamiltonians.
The $1$-periodic orbits of such an autonomous Hamiltonian are typically degenerate, so not admissible.
We will consider an autonomous Hamiltonian with the following type of degeneracy.
\begin{definition}
We say that the $1$-periodic orbits of an autonomous Hamiltonian $H$ are of {\bf Morse-Bott type} if
\begin{itemize}
\item the {\bf critical set} or {\bf critical manifold} $C=\{ x\in W~|~Fl^{X_H}_1(x)=x \}$ forms a (possibly disconnected) compact submanifold of $W$ without boundary; here $Fl^{X}_t$ denotes the time-$t$ flow of a vector field $X$.
\item the restriction of the linearized return map to the normal bundle of each connected component $\Sigma$ of $C$ is non-degenerate, i.e.~the linear map
$$
T_x Fl^{X_H}_1|_{\nu(\Sigma)}-\id |_{\nu(\Sigma)}
$$
is invertible for all $x\in \Sigma$.
\end{itemize}
\end{definition}
Now consider an autonomous Hamiltonian $H$ that only depends on the interval coordinate on the symplectization end, i.e.~$H|_{[1,\infty[\times \partial W}=h(r)$.
Then
$$
X_H=-h'(r)R_{\eta},
$$
so $1$-periodic orbits of $X_H$ correspond to Reeb orbits of $\eta:=\lambda|_{\partial W}$ with period $h'(r)$ (note that the $r$-coordinate is preserved under the flow of $X_H$).
Suppose furthermore that $\Sigma$ is a Morse-Bott manifold consisting of $1$-periodic orbits of $X_H$.
In the contact manifold $(\partial W,\lambda|_{\partial W} )$ this Morse-Bott manifold $\Sigma$ corresponds to a Morse-Bott manifold of $h'(r)$-periodic Reeb orbits.
Given an $h'(r)$-periodic Reeb orbit $x$, choose a capping disk $d_x$ in $\partial W$ together with a unitary trivialization of the contact structure $({ \mathcal D},J|_{{ \mathcal D}},d\lambda|_{{ \mathcal D}})$.
With respect to such a unitary trivialization, the restriction of the linearized Reeb flow to ${ \mathcal D}$ gives a path of symplectic matrices.
Define the Robbin-Salamon index of $x$ as the Robbin-Salamon index of this path of matrices.

\begin{remark}
This index is independent of the choice of $x$ in $\Sigma$. Moreover, it does not depend on the choice of the trivialization and the capping disk due to our assumption that $c_1(TW)=0$.
\end{remark}

Define the shift of $\Sigma$ as
$$
shift(\Sigma)=\mu_{RS}(\Sigma)-\frac{1}{2} (\dim \Sigma - 1),
$$
where $\mu_{RS}(\Sigma)$ is the Robbin-Salamon index of the \emph{Reeb} flow.
\subsubsection{Periodic flow}
Suppose now that $(M,\eta)$ is a quasi-regular contact manifold with a convenient filling $W$, and denote the periods of the Reeb flow by $T_1<\ldots<T_k$, where $T_k$ is the period of a principal orbit.

Since the action spectrum has the form $T_1\bbn \cup \ldots \cup T_k\bbn $, we can find $a\in \bbr^+$ such that each interval $[pa,(p+1)a]$ contains at most one critical value of $\mathcal A$ for every $p \in \bbn$.
Furthermore, we can assume that this critical value is an interior point of $[pa,(p+1)a]$.
Denote the set of critical manifolds with critical value in the interval $[pa,(p+1)a]$ by $C(p)$.

A Morse-Bott spectral sequence computing equivariant symplectic homology with more general, but more technical assumptions was proved in Appendix B of \cite{KwvKo13}.
The following, simplified version suffices for this paper.
\begin{theorem}[Morse-Bott spectral sequence for periodic flows]
\label{thm:SS_equivariant_homology}
Let $(W,\omega=d\lambda)$ be a convenient Liouville domain satisfying the following.
\begin{itemize}
\item $TW$ is trivial as a symplectic vector bundle.
\item The Reeb flow of $\partial W$ is periodic with minimal periods $T_1,\ldots, T_k$, where $T_k$ is the common period, i.e.~the period of a principal orbit.
\item There is a compatible complex structure on the contact manifold $(\partial W,\lambda_{\partial W})$ such that the linearized Reeb flow is complex linear.
\end{itemize}
Then there is also a spectral sequence converging to $SH^{+,S^1}(W;\bbq)$. Its $E^1$-page is given by
\begin{equation}
\label{eq:MB_SS_SH+S1}
E^1_{pq}(SH^{+,S^1})=
\begin{cases}
\bigoplus_{\Sigma \in C(p) } H_{p+q-shift(\Sigma)}^{S^1}(\Sigma;\bbq) & p>0 \\
0 & p\leq 0.
\end{cases}
\end{equation}
\end{theorem}
\begin{remark}
The first condition allows for a simple construction of the local coefficient system, and the third condition ensures that the local coefficient system is trivial.

Note that the natural filling of a Brieskorn manifold by its smoothed singularity satisfies all of the assumptions.
\end{remark}

\begin{lemma}[Equivariant symplectic homology as a contact invariant]
\label{lemma:equi_SH_invariant}
Assume that $(M,\eta)$ is a simply-connected quasi-regular contact manifold admitting a convenient filling $(W,\lambda)$.
Suppose furthermore that the spectral sequence~\eqref{eq:MB_SS_SH+S1} is {\bf lacunary}, meaning that
\begin{itemize}
\item for all $i>0$ and all $p,q\in \bbz$, the following holds. If $E^1_{p,q}\neq 0$, then $E^1_{p-i,q+i-1}=0$.
\end{itemize}
Under these assumptions $SH^{+,S^1}(W)$ does not depend on the choice of the convenient filling $W$. 
\end{lemma}
\begin{proof}
Suppose that $\tilde W$ is another convenient filling.
We claim that we can choose Liouville forms on $W$ and $\tilde W$ that coincide on the cylindrical ends $M\times [r_0,\infty[$.
To see this, note that if $\lambda$ and $\tilde \lambda$ are the restrictions of Liouville forms to the cylindrical end $M\times [r_0,\infty[$, then $d(\lambda-\tilde \lambda)=0$.
By the assumption that $M$ is simply-connected, we see that $\lambda-\tilde \lambda=df$.
Then define a new Liouville form on $M\times [r_0,\infty[$ by 
$$
\lambda':=\tilde \lambda+d(\rho \cdot f).
$$
Here $\rho$ is a cut-off function that vanishes on a neighborhood of $M\times \{ r_0\}$ and is equal to $1$ on the set $M\times [r_1,\infty[$.
This form $\lambda'$ extends to a globally defined Liouville form on $\tilde W$.

By Theorem~\ref{thm:SH_S1_sympl_def_invar}, different choices of Liouville forms yield isomorphic homologies.
With respect to the new Liouville form, the $E^1$-page of the Morse-Bott spectral sequence~\eqref{eq:MB_SS_SH+S1} computing $SH^{+,S^1}(W)$ is isomorphic to the $E^1$-page of the spectral sequence computing $SH^{+,S^1}(\tilde W)$.
The lacunarity assumption of the lemma then tells us that the spectral sequence abuts in both cases at the $E^1$-page.
Hence $SH^{+,S^1}(W)\cong SH^{+,S^1}(\tilde W)$.
\end{proof}

\subsection{Explicit formulae for the mean Euler characteristic}
In this section we follow \cite[Proposition 5.20]{KwvKo13} to get an explicit formula for the mean Euler characteristic (mec).
Suppose that $(M,\eta)$ is a quasi-regular contact manifold, and denote the periods of the Reeb flow by $T_1<\ldots<T_k$, where $T_k$ is the period of a principal orbit.
Define the function
\begin{equation}
\label{eq:phi_function}
\phi_{T_i;T_{i+1},\ldots,T_k} \,=\,
\# \{ a \in \bbn \mid aT_i < T_k \text{ and } a T_i \notin T_j \bbn \text{ for } j=i+1,\ldots, k \}
.
\end{equation}
We use the convention that $\phi_{T_k;\emptyset}=1$.
For $T=T_1,\ldots,T_k$, define
$$
\Sigma_T=\{
y\in M~|~Fl^{R_\eta}_{T}(y)=y
\}
.
$$
The equivariant Euler characteristic, i.e.~the Euler characteristic of the $S^1$-equivariant homology $H^{S^1}(\Sigma;\bbq)$, will be denoted by $\chi^{S^1}(\Sigma)$.
This equivariant Euler characteristic equals the Euler characteristic of the quotient if the circle action has no fixed points, which is the case here.
\begin{proposition}
\label{prop:mean_euler_S^1-orbibundle}
Let $(M,\eta)$ be a simply-connected quasi-regular contact manifold admitting a convenient filling $(W,d\lambda)$.
Suppose furthermore that the following conditions hold.
\begin{itemize}
\item The restriction of the tangent bundle to the symplectization of $M$, $T(\bbr \times M)|_{M}$, is trivial as a symplectic vector bundle.
\item There is a compatible complex structure on the contact manifold $M$ such that the linearized Reeb flow is complex linear.
\end{itemize}
Let $\mu_P:=\mu(M)$ denote the Maslov index of a principal orbit of the Reeb action.
If $\mu_P\neq 0$ then the following hold.
\begin{itemize}
\item The contact manifold $(M,\eta)$ is index-positive if $\mu_P>0$ and index-negative if $\mu_P<0$.
\item The mean Euler characteristic is an invariant of the contact structure and satisfies the following formula,
\begin{equation}
\label{eq:MEC_general}
\chi_m(W)=\frac{\sum_{i=1}^k (-1)^{\mu(\Sigma_{T_i})-\frac{1}{2}(\dim (\Sigma_{T_i})-1) } \phi_{T_i;T_{i+1},\ldots T_k} \chi^{S^1}(\Sigma_{T_i})}{|\mu_P|}.
\end{equation}
\end{itemize}
\end{proposition}

\subsubsection{Index formulae}
We follow the notation from \cite{KwvKo13}.
Suppose that $2\pi \cdot T$ is the minimal period of a periodic Reeb orbit in the Brieskorn manifold $L(a_0,\ldots,a_n)$.
Write $I=\{0,1,\ldots,n\}$ and $I_T=\{ j\in I~|~a_j \text{ divides }T \}$.
We obtain an entire Brieskorn submanifold $K(I_T):=L(\{ a_j\}_{j\in I_T})$ consisting of periodic orbits, which forms a Morse-Bott submanifold.

Then the Robbin-Salamon index of an $N$-fold cover is given by the following formula
\begin{equation}
\label{eq_Maslov_index_Brieskorn}
\mu(N\cdot K(I_{T})\,)=2\sum_{j \in I_{T}} \frac{NT}{a_j}+
2\sum_{j \in I-I_{T}} \bigg\lfloor \frac{NT}{a_j} \bigg\rfloor+\#(I-I_{T})
-2NT,
\end{equation}
provided that $a_j$ does not divide $NT$ for $j\in I-I_T$.
If the latter happens, the $N$-fold cover is part of a larger Morse-Bott submanifold of periodic orbits.

A principal orbit has period equal to $2\pi\cdot \lcm_i a_i$, so we see that the index of a principal orbit is given by
\begin{equation}
\label{eq:principal_orbit}
\mu_P=2\lcm_{j\in I} a_j \, \left( \sum_{j=0}^n\frac{1}{a_j} \, -1 \right) .
\end{equation}

\begin{remark}
We note that~\eqref{eq:principal_orbit} is twice the Fano index as defined in Definition 4.4.24 of \cite{BG05}. See also the proof of Theorem~11.7.8 in \cite{BG05}. A positive Maslov index of a principal orbit is hence equivalent to the quotient orbifold being log Fano.
\end{remark}

\section{Some formulae for the mean Euler characteristic for BP links}

Note that all conditions hold for Brieskorn manifolds of dimension $5$ or greater.
We work out all necessary ingredients.
The Maslov index of a principal orbit was already given in formula~\eqref{eq:principal_orbit}.
For each orbit space $\Sigma_T$ the Euler characteristic $\chi^{S^1}(\Sigma_T)$ can be computed by combining formula (3.4) and Theorem 3.10 from \cite{KwvKo13}. We have
\begin{equation}
\label{eq:rk_middle_dim_hom}
\rk H_{n-2}(L(a_0,\ldots,a_n);\bbq)
=\sum_{I_t \subset I_s} (-1)^{s-t}\frac{\prod_{i\in I_t}a_i}{\lcm_{j\in I_t}a_j}.
\end{equation}
Here $I_s=\{ 0,1,\ldots,n \}$ and $I_t$ denotes a subset with $t$ elements.

\subsubsection{A detailed example}
We work out one example in detail, and refer for the rest to the tables below.
Consider $L(2,3,4,4+12k)$.
The period of a principal orbit is $T=\lcm(2,3,4,4+12k)=3\cdot(4+12k)$.
By formulae of Milnor-Orlik \cite{MiOr70} and Randell \cite{Ran75} we find that $\chi^{S^1}(L(2,3,4,4+12k)\,)=\chi(L(2,3,4,4+12k)/S^1)=3$.
Note that this is also obvious from the rational Gysin sequence and the fact that $L(2,3,4,4+12k)$ is a rational homology sphere.

The exceptional orbits have periods dividing $3\cdot(4+12k)$. These periods are $12$ for $L(2,3,4)$, $6$ for $L(2,3)$ and $4$ for $L(2,4)$.
We can again use Milnor-Orlik to find the relevant Euler characteristics, but may also observe that $L(2,3,4)$ is log Fano and hence the quotient must be homeomorphic to $S^2$, which has $\chi=2$.
The Brieskorn manifolds $L(2,3)$ is the $2,3$-torus knot, so $\chi=1$ and $L(2,4)$ is a torus link with two components.

We now compute the values of the $\phi$-function, defined in~\eqref{eq:phi_function}.
The first one equals by definition $\phi_{3\cdot(4+12k);\emptyset}=1$.
We work backwards to get the others
\[
\begin{split}
\phi_{12;3\cdot(4+12k)}&=3\cdot(4+12k)/12-1=1+3k-1=3k.
\\
\phi_{6;12,3\cdot(4+12k)}&=3\cdot(4+12k)/6-1-3k=2+6k-1-3k=1+3k.
\\
\phi_{4;12,3\cdot(4+12k)}&=3\cdot(4+12k)/4-1-3k=1+9k-1-3k=6k.
\end{split}
\]
Finally the Maslov index of a principal orbit is
$$
2\cdot 3\cdot(4+12k)\cdot(\frac{1}{2}+\frac{1}{3}+\frac{1}{4}+\frac{1}{3\cdot(4+12k)}-1)
=6+8k.
$$
We collect the above data in a table.
\[
\begin{tabular}{llll}
Orbit space &  period & $\chi^{S^1}$ & frequency (in one period of $E^1$) \\
\hline
$L(2,3,4,4+12k)$ & $2^2\cdot 3\cdot(1+3k)$ & 3 & 1 \\
$L(2,3,4)$ & $12$ & $2$ & $3k$ \\
$L(2,3)$ & $6$ & 1 & $1+3k$ \\
$L(2,4)$ & $4$ & 2 & $6k$ \\
\end{tabular}
\]
Combine this with the proposition and we find
$$
\chi_m=\frac{6k\cdot 2+(1+3k)\cdot 1+3k\cdot 2+1\cdot 3}{6+8k}
=\frac{4+21k}{6+8k}.
$$
Note that the above computation does not apply to the case $k=0$. In the latter case, there are only three orbits spaces, namely $L(2,3,4,4)$, $L(2,3)$ and $L(2,4)$.
The formula for the mean Euler characteristic is still correct though.
The others follow the same procedure.

In Sections 6.1-3 below we make use of the classification of positive BP links in dimension 5 given in Tables B.4 of \cite{BG05}.

\subsection{Positive BP links diffeomorphic to $S^5$}
All positive BP links on $S^5$ were given in Table B.4.3 of \cite{BG05}. As mentioned in the Introduction, it was already known from \cite{Ust99,KwvKo13,Gutt15} that $|\pi_0(\gM^c_{+,0}(S^5))|=\aleph_0$, so here we are content to illustrate one infinite series where the mean Euler characteristic is given. The notational convention employed here as well as in the next section is to present a column of three entries, the first of which gives the manifold, the second the Brieskorn link, and third the mean Euler characteristic. We have

\begin{equation}
\begin{alignedat}{2}
S^5 &\quad
L(2,3,5,1+30k) &\quad
\frac{31+270k}{62+60k}. \\
\end{alignedat}
\end{equation}

Notice that it follows from Theorem \ref{Lichbd} that $L(2,3,5,1+30k)$ cannot admit an SE metric when $k>1$. For $k=1$ it is unknown whether it has an SE metric.

\begin{example}\label{S5SEex2}
We continue our discussion of Example \ref{S5SEex} concerning SE metrics on $S^5$. We have computed the mean Euler characteristic for these families and presented them in an Excel table which can be found in \url{http://www.math.snu.ac.kr/~okoert/tools/BP5_list_full_SH.xls}. The table also contains the {\it number of moduli}, that is, the complex dimension of the local moduli space computed from Proposition \ref{hyslocmod}. Notice that our Excel table shows that there are 7 pairs where the mean Euler characteristic does not distinguish the contact structures. In the case of $L(2,3,7,22)$ and $L(3,3,4,7)$ the mean Euler characteristics coincide, but in Lemma~\ref{lemma:examples_differ} below we show that their contact structures are not isomorphic. In the remaining 6 pairs we cannot say whether their contact structures are inequivalent or not. Hence, we see that there are at least $76$ inequivalent contact structures on $S^5$ containing an SE metric implying $|\pi_0(\gM^{SE}(S^5))|\geq 76.$ This proves Theorem \ref{s5SEthm} of the Introduction.
\end{example}

\begin{lemma}
\label{lemma:examples_differ}
The contact manifolds $L(2,3,7,22)$ and $L(3,3,4,7)$ are not contactomorphic.
\end{lemma}
\begin{proof}
We write out the Morse-Bott spectral sequence~\ref{thm:SS_equivariant_homology} and apply Lemma~\ref{lemma:equi_SH_invariant} to verify this.
Both spectral sequences are lacunary, so $SH^{+,S^1}_k=\oplus_{p+q=k} E^1_{pq}$ are invariants of the contact structure.
Note that $\rk SH^{+,S^1}_0( V(2,3,7,22);\bbq)=6$, whereas
$$\rk SH^{+,S^1}_0( V(3,3,4,7);\bbq)=7.$$
\end{proof}

\subsection{Positive BP links in dimension 5 that are rational homology spheres and not diffeomorphic to $S^5$} 
The positive BP links in dimension 5 that are nontrivial rational homology spheres are classified in Table B.4.1 in \cite{BG05}. Here we list those with a countable number of Sasaki-Seifert structures together with their mean Euler characteristic. 

\begin{align*}
\begin{alignedat}{2}
M_2 &\quad
L(2,3,3,3+6k) &\quad
\frac{3+10k}{6+4k} \\
M_3 &\quad
L(2,3,4,4+12k) &\quad
\frac{4+21k}{8+6k} \\
 &\quad
L(2,3,4,8+12k) &\quad
\frac{11+21k}{10+6k} \\
\end{alignedat}
\end{align*}

\begin{align*}
\begin{alignedat}{2}
M_5 &\quad
L(2,3,5,6+30k) &\quad
\frac{6+45k}{12+10k} \\
&\quad
L(2,3,5,12+30k) &\quad
\frac{15+45k}{14+10k} \\
&\quad
L(2,3,5,18+30k) &\quad
\frac{24+45k}{16+10k} \\
&\quad
L(2,3,5,24+30k) &\quad
\frac{33+45k}{18+10k} \\
\end{alignedat}
\end{align*}

\begin{align*}
\begin{alignedat}{2}
2M_3 &\quad
L(2,3,5,10+30k) &\quad
\frac{4+27k}{8+6k} \\
&\quad
L(2,3,5,20+30k) &\quad
\frac{13+27k}{10+6k} \\
4M_2 &\quad
L(2,3,5,15+30k) &\quad
\frac{3+18k}{6+4k} \\
\end{alignedat}
\end{align*}

It follows that these rational homology spheres have a countable infinity of inequivalent contact structures of positive Sasaki type. This proves Theorem \ref{rathomsphcomp} of the Introduction. Note that if $k>1$ Theorem \ref{Lichbd} says that these links cannot admit an SE metric. It is unknown when $k=1$. Nevertheless, all of the above rational homology spheres with the exception of $4M_2$ are known to  admit SE metrics \cite{BG05}. Whether $4M_2$ admits an SE metric has so far proven to be quite elusive. 

\subsection{Positive BP links on the connected sums $k(S^2\times S^3)$}
Brieskorn manifolds diffeomorphic to $(n-1)S^2\times S^3$ are given by the link $L(2,2,p,q)$ with $n=\gcd(p,q)$. Here $n=1$ means $S^5$. These generally do not admit SE metrics, however.  

\begin{lemma}\label{ks2s3}
The Brieskorn manifolds $L(2,2,p,q)\approx (n-1)(S^2\times S^3)$ where $n=\gcd(p,q)$ satisfy
$$\chi_m(L(2,2,p,q))= \frac{pq+n^2}{2(p+q)}.$$
\end{lemma}

Thus, for each $n>1$ the 5-manifolds $(n-1)(S^2\times S^3)$ have a countably infinite number of inequivalent contact structures. This proves Theorem \ref{ks2s3comp} of the Introduction.

Concerning SE metrics, it is known \cite{BG05} that there are infinitely many Sasaki-Seifert structures admitting SE metrics on the manifolds $k(S^2\times S^3)$ for $k>2$. However, these do not occur on Brieskorn manifolds. 

\subsection{Homotopy $7$-Spheres with Sasaki-Einstein Metrics}\label{homotopy7sphsec}
Below is a table which gives the number of Brieskorn homotopy $7$-spheres that admit Sasaki-Einstein metrics and which can be distinguished by the mean Euler Characteristic of a Liouville filling. The existence of pairs indicates that the mean Euler characteristic of two coincide. There is also one case of a triple which occurs with signature 15. $N$ is the number of Brieskorn homotopy $7$-spheres with the indicated signature. It also indicates the number of Sasaki-Seifert structures. The last column gives the number of pairs with the same mean Euler characteristic. From this one easily obtains lower bounds on $|\pi_0(\gM^{SE}(\grS^7))|$. The table is based on the Excel file that can be found in \url{http://www.math.snu.ac.kr/~okoert/tools/BP7_list_same_mec_exo.xls}. In the table the Milnor signature, indicating the oriented homotopy $7$-sphere from $0,\ldots,27$, is given by the column labeled exo. 
\bigskip

\centerline{Oriented Homotopy $7$-spheres with SE metrics}
\bigskip
\begin{center}
\begin{tabular}{| l | r | r || r | r | r |}
\hline
sig & $N$ & pairs & sig & N & pairs \\ \hline
0 & 353 &      0     & 14  & 390 & 1   \\ \hline
1 & 376 &     0      & 15  & 409 & 0   \\ \hline
2 & 336 & 2          & 16  & 352 & 3   \\ \hline
3 & 260 & 1          & 17  & 226 & 1   \\ \hline
4 & 294 & 1          & 18  & 260 & 0   \\ \hline
5 & 231 & 4          & 19  & 243 & 0   \\ \hline
6 & 284 & 2          & 20  & 309 & 1   \\ \hline
7 & 322 & 1          & 21  & 292 & 1   \\ \hline
8 & 402 & 2          & 22  & 425 & 1   \\ \hline
9 & 317 & 1          & 23  & 307 & 2   \\ \hline
10&309 & 5          & 24  & 298 & 0   \\ \hline
11&252 & 2          & 25  & 230 & 1   \\ \hline
12&304 & 0          & 26  & 307 & 2   \\ \hline
13&258 & 0          & 27  & 264 & 0   \\ \hline

\end{tabular}
\end{center}
\medskip

These are not true lower bounds since we have not accounted for the SE metrics on the homotopy sphere $S^7$ in \cite{GhKo05}. However, for these metrics the diffeomorphism type has not been determined.

\subsection{Homotopy $9$-Spheres with Sasaki-Einstein Metrics}\label{homotopy9sphsec}
For homotopy spheres of dimension $9$ that admit an SE metric, we present an Excel table that can be found in \url{http://www.math.snu.ac.kr/~okoert/tools/BP9_list984.xls}. One can use this table and the invariance of the mean Euler characteristic to distinguish the contact structures. The column labelled sig indicates which homotopy $9$ in $bP_{10}$, $1$ indicates the standard sphere, and $3$ the exotic Kervaire sphere. This proves Theorem \ref{s9SEthm}.

\subsection{Higher dimensional Diffeomorphism Types}
Next we consider links studied in \cite{BG05h} and discussed in Section 9.5.2 of \cite{BG05}. These do not admit SE metrics, but they are probably easier to work with. First we treat dimension 1 mod 4. The first case is
$$f=z_0^{8l}+z_1^2+\cdots +z_{2n+1}^2=0.$$
The link $L(8l,2,\ldots,2)$ has diffeomorphism type $S^{2n}\times S^{2n+1}$. Next we have the link $L(8l+4,2,\ldots,2)$ with
$$f=z_0^{8l+4}+z_1^2+\cdots +z_{2n+1}^2=0$$
whose diffeomorphism type is $S^{2n}\times S^{2n+1}\#\grS^{4n+1}$ where $\grS^{4n+1}$ is a generator of the group $bP_{4n+2}$. Note that $bP_{4n+2}\approx \bbz_2$ when $n\neq 2^i-1$ for some $i=1,2,\ldots$, so $\grS^{4n+1}$ is exotic is in this case.  The last case of this type is $L(4l+2,2,\ldots,2)$ given by the polynomial
$$f=z_0^{4l+2}+z_1^2+\cdots +z_{2n+1}^2=0.$$
Here the diffeomorphism type is the unit tangent sphere bundle $T$ of $S^{2n+1}$. There are relations such as $T\# \grS^{4n+1}$ is diffeomorphic to $T$, and we can consider more general connected sums; however, it is often difficult to determine the precise diffeomorphism type.

\begin{proposition}\label{infcon}
There are a countably infinite number of inequivalent contact structures of Sasaki type on the manifolds $S^{2n}\times S^{2n+1},S^{2n}\times S^{2n+1}\#\grS^{4n+1}$ and $T$, where $\grS^{4n+1}$ is a generator of the group $bP_{4n+2}$. 
\end{proposition}

\begin{proof}
We distinguish the contact structures in the three cases by using the formula for the mean Euler characteristic given in Appendix A of \cite{KwvKo13}.
\end{proof}

Proposition \ref{infcon} proves Theorem \ref{highdim} of the Introduction.

\begin{remark}
Since these Brieskorn manifolds have a Sasaki automorphism group which contains $SO(2n+1)$ owing to the $2n+1$ terms with exponent equal to $2$, they have a Sasaki cone of dimension $n+1$. One can then ask whether there is a deformation in the Sasaki cone to a Sasaki-Einstein metric? The answer is no due to the work of Martelli, Sparks, Yau \cite{MaSpYau06} and He \cite{He14}.
\end{remark}

\section{Different moduli of Sasaki-Einstein metrics on spheres using the Sylvester sequence}\label{Sylsec}
Define the Sylvester sequence $\{ c_i \}_{i=0}^\infty$ recursively by
\[
c_i=c_{i-1}\cdot(c_{i-1}-1)+1,\quad c_0=2.
\]
We have $c_i=1+\prod_{j=0}^{i-1}c_j$, and we see that the $c_i$'s are pairwise relatively prime.
We will consider the Brieskorn-Pham links 
$$
L^{2n+3}(a)=L(2,2c_0,\ldots,2c_n,a),
$$
where $a$ is relatively prime to the $c_i$'s.
By the proof of Theorem~11.5.5 from \cite{BG05} these manifolds carry Sasaki-Einstein metrics.

To compute the mean Euler characteristic, we need to identify the orbit spaces of the Reeb action.
\begin{itemize}
\item The principal orbits are the full space $L(2,2c_0,\ldots,2c_{n},a)$.
\item The strata of exceptional orbits of dimension $2k+1$ (so of codimension $2n-2k+2$) are formed by the Brieskorn submanifolds
$L(2,2c_{i_0},\ldots, 2c_{i_k})$ and 
$L(2,2c_{i_0},\ldots,2c_{i_{k-1}},a)$.
\end{itemize}
In the following lemma we compute the equivariant Euler characteristics of Brieskorn submanifolds.
\begin{lemma}
\label{lemma:euler_chars_sylvester}
If $n$ is odd, then $L^{2n+3}(a)$ is homeomorphic to $S^{2n+3}$.
Furthermore, if $i_0<\ldots<i_k$, then the Brieskorn submanifolds \linebreak $L(2,2c_{i_0},\ldots,2c_{i_k})$ and $L(2,2c_{i_0},\ldots,2c_{i_k},a)$
satisfy
\begin{equation}
\label{eq:chi_type_no_a}
\chi^{S^1}(L(2,2c_{i_0},\ldots,2c_{i_k})\,)=k+2+\frac{1}{2}(1-(-1)^{k+1})
\end{equation}
\begin{equation}
\label{eq:chi_type_a}
\chi^{S^1}(L(2,2c_{i_0},\ldots,2c_{i_k},a)\,)=k+3.
\end{equation}
In particular, the equivariant Euler characteristics are independent of $a$ provided $a$ satisfies the relatively prime condition.
\end{lemma}
\begin{proof}
The first claim follows from the Brieskorn graph theorem, see Theorem~9.3.18 ii) from \cite{BG05}.
Just note that $a$ is an isolated point and that there is an odd number of vertices $\{2,2c_0,\ldots,2c_n\}$ with pairwise $gcd$ equal to $2$.

Formula~\eqref{eq:chi_type_a} is also proved with the Brieskorn graph theorem, namely Theorem~9.3.18 i) from \cite{BG05} tells us that $L(2,\{ 2c_{i_j} \}_{j},a)$ is a rational homology sphere since the vertex $\{a \}$ is an isolated point.
The claim about the Euler characteristic then follows from Lemma~4.5 in \cite{FSvK12}.

To prove formula~\eqref{eq:chi_type_no_a}, we use a formula from Milnor-Orlik \cite{MiOr70} (cf. Corollary 9.3.13 of \cite{BG05}).
We compute $\kappa=\rk H_{k}(L(2,2c_{i_0},\ldots,2c_{i_k});\bbz)$,
\begin{align*}
(-1)^{k+2}\kappa&=
1-(k+2)+
\sum_{\ell=2}^{k+1} (-1)^\ell
\sum_{I_\ell\subset I} \frac{\prod_{i\in I_\ell}a_i}{\lcm_{i\in I_\ell}a_i}
+(-1)^{k+2} \frac{2^{k+2}}{2}
\\
&=
-(k+1)
+\sum_{\ell=2}^{k+1}(-1)^\ell 2^{\ell-1} { {k+1}\choose{\ell-1} } \\
&~~
+
\sum_{\ell=2}^{k+1}(-1)^{\ell} 2^{\ell-1}
{ {k+1}\choose{\ell} }
+(-1)^{k+2} \frac{2^{k+2}}{2} \\
&=-(k+1)
+\sum_{\ell^\prime=1}^k (-1)^{\ell^\prime+1}2^{\ell'}
{ {k+1}\choose{\ell^\prime} }\\
&~~
+
\sum_{\ell=2}^{k+1}(-1)^{\ell} 2^{\ell-1}
{ {k+1}\choose{\ell} }
+(-1)^{k+2} \frac{2^{k+2}}{2} \\
&=
-(k+1)+\sum_{\ell=2}^k(-1)^{\ell+1}(2^\ell-2^{\ell-1})
{ {k+1}\choose{\ell} } \\
&~~~
+2(k+1)+(-1)^{k+1} \frac{2^{k+1}}{2}
+(-1)^{k+2} \frac{2^{k+2}}{2} \\
&=
-\frac{1}{2}\sum_{\ell=0}^{k}(-1)^\ell 2^\ell 
{ {k+1}\choose{\ell} }
+(-1)^{k+2} \frac{2^{k+1}}{2}+\frac{1}{2}\\
&
=\frac{1}{2}(1-(1-2)^{k+1} )
=\frac{1}{2}(1-(-1)^{k+1}).
\end{align*}
Combine this with Theorem~1 from \cite{Ran75} to obtain formula~\eqref{eq:chi_type_no_a}.
\end{proof}

To finish the computation of the numerator of Formula~\eqref{eq:MEC_general}, we determine the $\phi$-functions, \eqref{eq:phi_function}, by downward induction.
The principal orbits $L(2,2c_0,\ldots,2c_n,a)$ occur once.
The exceptional orbits $L(2,2c_{i_0},\ldots,2c_{i_k})$ occur 
$$
(a-1)\prod_{j=0, j\notin\{ i_0,\ldots,i_\ell\} }^n
(c_j-1)
$$
times, and the exceptional orbits $L(2,2c_{i_0},\ldots,2c_{i_k},a)$ occur
$$
\prod_{j=0, j\notin\{ i_0,\ldots,i_\ell\} }^n
(c_j-1)
$$
many times without being part of a larger orbit space.

We conclude that numerator of $\chi_m(L^{2n+3}(a))$ is given by
\[
\begin{split}
\chi^{2n+1}(a) &=
\sum_{\ell=0}^n \sum_{i_0<\ldots<i_\ell}
\left(
\prod_{j=0, j\notin\{ i_0,\ldots,i_\ell\} }^n
(c_j-1)
\right)
(a-1)
\chi^{S^1}(L(2,2c_{i_0},\ldots,2c_{i_\ell})\, ) \\
&~~~~+
\sum_{\ell=1}^n \sum_{i_0<\ldots<i_\ell}
\left(
\prod_{j=0, j\notin\{ i_0,\ldots,i_\ell\} }^n
(c_j-1)
\right)
\chi^{S^1}(L(2,2c_{i_0},\ldots,2c_{i_\ell},a)\, ).
\end{split}
\]
By Lemma~\ref{lemma:euler_chars_sylvester} we see that coefficient of $(a-1)$ is positive, and the second term is independent of $a$, so this function is a linearly increasing in $a$.
Furthermore, from Formula~\eqref{eq:principal_orbit} we see that the function $a\mapsto 1/\mu_P(a)$ is strictly increasing as long as $a$ is relatively prime to the $c_i$'s and the log Fano condition $\mu_P>0$ holds.
It follows that the function $a\mapsto \chi_m(a)$ is an injective function of $a$.

The proof of Theorem~11.5.5 from \cite{BG05} hence implies
\begin{proposition}
The contact manifolds $(L^{2n+3}(a),{\mathcal D}_a)$ are pairwise non-isomorphic (as contact manifolds).
In particular, for $n=3,5,7,\ldots$ the number of components of the moduli space of Sasaki-Einstein metrics on either $S^{2n+3}$ or the Kervaire sphere $\Sigma^{2n+3}$ is growing doubly exponentially with $n$.
\end{proposition}

\def\cprime{$'$} \def\cprime{$'$} \def\cprime{$'$} \def\cprime{$'$}
  \def\cprime{$'$} \def\cprime{$'$} \def\cprime{$'$} \def\cprime{$'$}
  \def\cdprime{$''$} \def\cprime{$'$} \def\cprime{$'$} \def\cprime{$'$}
  \def\cprime{$'$}
\providecommand{\bysame}{\leavevmode\hbox to3em{\hrulefill}\thinspace}
\providecommand{\MR}{\relax\ifhmode\unskip\space\fi MR }
\providecommand{\MRhref}[2]{%
  \href{http://www.ams.org/mathscinet-getitem?mr=#1}{#2}
}
\providecommand{\href}[2]{#2}

\end{document}